\documentclass[draft]{amsart}
\usepackage{amsthm}
\usepackage[leqno]{amsmath}
\usepackage{latexsym,amsfonts,amssymb}
\usepackage[all]{xy} \SelectTips{eu}{}
\usepackage{hyperref}
\usepackage{enumerate,verbatim}

\usepackage[usenames,dvipsnames]{pstricks}
\usepackage[mathscr]{eucal}
\usepackage{amsfonts,amsmath,amssymb,amsthm,amscd,amsxtra}

\newcommand{\numberseries}{\bfseries}   

\newlength{\thmtopspace}                
\newlength{\thmbotspace}                
\newlength{\thmheadspace}               
\newlength{\thmindent}                  

\newtheoremstyle{bfupright head,slanted body}
                {\thmtopspace}{\thmbotspace}
                {\slshape}{\thmindent}{\bfseries}{.}{\thmheadspace}
                {{\numberseries \thmnumber{#2\;}}\thmnote{#3}}

\newtheoremstyle{bfupright head,upright body}
                {\thmtopspace}{\thmbotspace}
                {\upshape}{\thmindent}{\bfseries}{.}{\thmheadspace}
                {{\numberseries \thmnumber{#2\;}}\thmnote{#3}}

\newtheoremstyle{fixed bf head,slanted body}
                {\thmtopspace}{\thmbotspace}{\slshape}
                {\thmindent}{\bfseries}{.}{\thmheadspace}
                {{\numberseries \thmnumber{#2\;}}\thmname{#1}\thmnote{ (#3)}}

\newtheoremstyle{fixed bf head,upright body}
                {\thmtopspace}{\thmbotspace}{\upshape}
                {\thmindent}{\bfseries}{.}{\thmheadspace}
                {{\numberseries \thmnumber{#2\;}}\thmname{#1}\thmnote{ (#3)}}

\newtheoremstyle{numbered paragraph}
                {\thmtopspace}{\thmbotspace}{\upshape}
                {\thmindent}{\upshape}{}{\thmheadspace}
                {{\numberseries \thmnumber{#2.}}}

\theoremstyle{bfupright head,slanted body}
\newtheorem{res}{}[section]

\theoremstyle{bfupright head,upright body}
\newtheorem{bfhpg}[res]{}               \newtheorem*{bfhpg*}{}

\theoremstyle{fixed bf head,slanted body}
\newtheorem{thm}[res]{Theorem}          \newtheorem*{thm*}{Theorem}
\newtheorem{prp}[res]{Proposition}      \newtheorem*{prp*}{Proposition}
\newtheorem{cor}[res]{Corollary}        \newtheorem*{cor*}{Corollary}
            \newtheorem*{lem*}{Lemma}

\theoremstyle{fixed bf head,upright body}
       \newtheorem*{dfn*}{Definition}
\newtheorem{rmk}[res]{Remark}           \newtheorem*{rmk*}{Remark}
\newtheorem{exa}[res]{Example}          \newtheorem*{exa*}{Example}
\newtheorem{que}[res]{Question}          \newtheorem*{que*}{Question}

\theoremstyle{numbered paragraph}
\newtheorem{ipg}[res]{}

\newenvironment{prf*}[1][Proof]{%
  \begin{proof}[\bf #1]
    \setcounter{equation}{0}
    }
  {\end{proof}
}

\newcommand{\pgref}[1]{\ref{#1}}
\newcommand{\thmref}[2][Theorem~]{#1\pgref{thm:#2}}
\newcommand{\corref}[2][Corollary~]{#1\pgref{cor:#2}}
\newcommand{\prpref}[2][Proposition~]{#1\pgref{prp:#2}}

\newcommand{\exaref}[2][Example~]{#1\pgref{exa:#2}}
\renewcommand{\eqref}[1]{(\pgref{eq:#1})}
\newcommand{\thmcite}[2][?]{\cite[thm.~#1]{#2}}
\newcommand{\rmkcite}[2][?]{\cite[rmk.~#1]{#2}}

\def\urltilda{\kern -.15em\lower .7ex\hbox{\~{}}\kern .04em}
\newcommand{\setof}[3][\mspace{1mu}]{\{#1#2 \mid #3#1\}}
\newcommand{\ZZ}{\mathbb{Z}}
\newcommand{\CC}{\mathbb{C}}
\newcommand{\deq}{\:=\:}
\newcommand{\dge}{\:\ge\:}

\newcommand{\dis}{\:\is\:}
\newcommand{\dqis}{\:\qis\:}

\newcommand{\Supp}{\operatorname{Supp}}

\newcommand{\ttpp}[2]{{\boldsymbol\top}_{#1}{#2}}

\newcommand{\Tr}{\mathsf{Tr}}
\newcommand{\m}{\mathfrak{m}}
\newcommand{\p}{\mathfrak{p}}
\newcommand{\is}{\cong}
\newcommand{\qis}{\simeq}
\renewcommand{\le}{\leqslant}
\renewcommand{\ge}{\geqslant}
\newcommand{\lra}{\longrightarrow}
\newcommand{\xra}[2][]{\xrightarrow[#1]{\;#2\;}}
\newcommand{\qra}{\xra{\qis}}

\newcommand{\Ker}[1]{\nobreak{\operatorname{Ker}#1}}
\newcommand{\Coker}[1]{\nobreak{\operatorname{Coker}#1}}
\renewcommand{\H}[2][]{\operatorname{H}_{#1}(#2)}
\newcommand{\Shift}[2][]{\mathsf{\Sigma}^{#1}{#2}}
\newcommand{\Tsb}[2]{#2_{{\scriptscriptstyle\supset}#1}}
\newcommand{\dptR}{\operatorname{depth}R}
\newcommand{\dpt}[2][R]{\operatorname{depth}_{#1}#2}
\newcommand{\pd}[2][R]{\operatorname{pd}_{#1}#2}
\newcommand{\Hom}[3][R]{\operatorname{Hom}_{#1}(#2,#3)}
\newcommand{\RHom}[3][R]{\operatorname{\mathbf{R}Hom}_{#1}(#2,#3)}
\newcommand{\tp}[3][R]{\nobreak{#2\otimes_{#1}#3}}
\newcommand{\tpp}[3][R]{(\tp[#1]{#2}{#3})}

\newcommand{\Tor}[4][R]{\operatorname{Tor}^{#1}_{#2}(#3,#4)}
\newcommand{\Ext}[4][R]{\operatorname{Ext}_{#1}^{#2}(#3,#4)}
\newcommand{\TExt}[4][R]{\operatorname{\widehat{Ext}}_{#1}^{#2}(#3,#4)}
\DeclareMathOperator{\Ass}{Ass}
\newcommand{\G}{\mathcal{G}}
\newcommand{\Gdim}[2][R]{\operatorname{\G-dim}_{#1}#2}
\newcommand{\Ttor}[4][R]{\smash{\operatorname{\widehat{Tor}}%
  }_{#2}^{#1^{\phantom{|\mspace{-6mu}}}}(#3,#4)}
\newcommand{\GLtp}[3][R]{\nobreak{#2\otimes_{#1}^{\mathbf{GL}}#3}}
\newcommand{\GLtpp}[3][R]{(\GLtp[#1]{#2}{#3})}
\newcommand{\GTor}[4][R]{\operatorname{\G Tor}^{#1}_{#2}(#3,#4)}

\numberwithin{equation}{res}

\begin{document}

\title[Vanishing of relative homology and depth of tensor products ]
{Vanishing of relative homology and depth of tensor products}
\author[Celikbas, Liang, Sadeghi]
{Olgur Celikbas, Li Liang and Arash Sadeghi}

\address{Olgur Celikbas\\
Department of Mathematics \\
West Virginia University\\
Morgantown, WV 26506-6310, U.S.A}
\email{olgur.celikbas@math.wvu.edu}

\address{Li Liang\\
School of Mathematics and Physics, Lanzhou Jiaotong University, Lanzhou 730070, China.}
\email{lliangnju@gmail.com}

\address{Arash Sadeghi\\
School of Mathematics, Institute for Research in Fundamental Sciences, (IPM), P.O. Box:
19395-5746, Tehran, Iran.}
\email{sadeghiarash61@gmail.com}

\thanks{}



\begin{abstract}
  For finitely generated modules $M$ and $N$ over a Gorenstein local
  ring $R$, one has $\dpt[]{M} + \dpt[]{N}= \dpt[]{\tpp{M}{N}} +
  \dptR$, i.e.,  the depth formula holds, if $M$ and $N$ are Tor-independent and Tate homology $\Ttor[]{i}{M}{N}$ vanishes for all $i\in\ZZ$. We establish the same conclusion under weaker hypotheses: if $M$ and $N$ are $\G$-relative Tor-independent, then the vanishing of $\Ttor[]{i}{M}{N}$ for all $i\le 0$ is enough for the depth formula to hold. We also analyze the depth of tensor products of modules and obtain a necessary condition for the depth formula in terms of $\G$-relative homology.
\end{abstract}

\ \vspace{-1\baselineskip}
\maketitle

\thispagestyle{empty}

\section{Introduction}

\noindent
In this paper, $R$ is a commutative Noetherian local ring with unique maximal ideal $\mathfrak{m}$ and residue field $k$, and
$R$-modules are tacitly assumed to be finitely generated.

In 1961 Auslander \cite[1.2]{Au} proved a natural extension of the classical Auslander -- Buchsbaum formula. He showed that, if $M$ has finite projective dimension (e.g., $R$ is regular) and $M$ and $N$ are Tor-independent modules, i.e., $\Tor{i}{M}{N}=0$ for all $i\ge 1$, then the following remarkable equality holds:
\begin{equation*}\tag{1}
\dpt{N}= \dpt{\tpp{M}{N}} + \pd{M}.
\end{equation*}
Notice, for the special case where $N=R$, equality (1) recovers the  Auslander -- Buchsbaum formula, i.e., $\dpt{R}= \dpt{M} + \pd{M}$.

In 1994 Huneke and Wiegand established an important consequence of Tor-independence. They showed equality (1) holds for Tor-independent modules $M$ and $N$ over complete intersection rings $R$ (even if $M$ does not necessarily have finite projective dimension) provided that $\pd{M}$ is replaced by $\dpt{R}-\dpt{M}$. More precisely, Huneke and Wiegand \cite[2.5]{HW} proved, if $R$ is a complete intersection and $M$ and $N$ are Tor-independent modules, then one has:
\begin{equation*}\tag{2}
\dpt{M} + \dpt{N} = \dpt{R}+\dpt{\tpp{M}{N}}.
\end{equation*}
The equality in (2) is dubbed \emph{the depth formula} by Huneke and Wiegand in \cite{HW2}.

The depth formula is an important tool to study depth of tensor products of modules as well as that of complexes \cite{SIn99}. For example if $M$ and $N$ are maximal Cohen-Macaulay modules (i.e., $\dpt{M}= \dpt{N} =\dim{R})$ and the depth formula holds, then $R$ must be Cohen-Macaulay and $M\otimes_{R}N$ is maximal Cohen-Macaulay. In general it is an open question, even over Gorenstein rings, whether or not the depth formula holds for all Tor-independent modules.

There are quite a few extensions of the aforementioned result of Huneke and Wiegand in the literature; see for example \cite{Dao}.  A substantial generalization in this direction  is due to Christensen and Jorgensen: if $M$ has finite $\G$-dimension in the sense of Auslander and Bridger~\cite{MAsMBr69} (e.g., $R$ is Gorenstein), then the vanishing of Tate homology $\Ttor{i}{M}{N}$ for all $i\in \ZZ$ (e.g., $M$ has finite projective dimension) is a sufficient condition for the \emph{derived depth formula} to hold; see \cite[2.3]{LWCDAJ15}. Noting that the derived depth formula coincides with the depth formula for Tor-independent modules $M$ and $N$, we observe:

\begin{thm} (Christensen and Jorgensen; see \cite[5.3]{LWCDAJ15}) \label{thm:recover}
Let $M$ and $N$ be $R$-modules such that $M$ has finite $\G$-dimension. Then the depth formula holds, i.e., $\dpt{M} + \dpt{N} = \dpt{\tpp{M}{N}} + \dptR$, provided the following conditions hold.
\begin{enumerate}[\rm(i)]
\item $\Ttor{i}{M}{N}=0$ for all $i\in\ZZ$.
\item $\Tor{i}{M}{N}=0$ for all $i\ge 1$.
\end{enumerate}
\end{thm}

The main aim of this article is to obtain a new condition that is sufficient for the depth formula to hold. A new tool we use is the \emph{$\G$-relative homology} $\GTor{*}{M}{N}$ which has been defined and studied by Avramov and Martsinkovsky~\cite{LLAAMr02}, and Iacob~\cite{AIc07}; see \ref{RH} for the definition of relative homology. We seek to find out how the vanishing of $\G$-relative homology relates to depth of tensor products of modules. One of the consequences of our main argument, \thmref{gdepth}, is an extension of \thmref{recover}. More precisely we obtain the following result in \corref{general}.

\begin{thm} \label{thm:gdepth1} Let $M$ and $N$ be $R$-modules such that $M$ has finite $\G$-dimension. Then the depth formula holds provided the following conditions hold.
\begin{enumerate}[\rm(i)]
\item $\Ttor{i}{M}{N}=0$ for all $i \le 0$.
\item $\GTor{i}{M}{N}=0$ for all $i \ge 1$.
\end{enumerate}
\end{thm}

Let us remark that the hypotheses of \thmref{gdepth1} are in general weaker than those of  \thmref{recover}: Jorgensen and \c{S}ega \cite[4.1]{DAJLMS05} constructed modules $M$ and $N$ with $\Ttor{i}{M}{N}=0$ for all $i\le 0$, $\GTor{i}{M}{N}=0$ for all $i\ge 1$ and $\Ttor{i}{M}{N}\ne 0$ for all $i\ge 2$; see \exaref{2}. Moreover relative homology vanishes more frequently than absolute homolgy. For example, if $M$ is totally reflexive, i.e., $M$ has $\G$-dimension zero, then $\GTor{i}{M}{N}=0$ for all $i\ge 1$; see \ref{RH}. In particular this establishes, by \thmref{gdepth1}, if $M$ is totally reflexive and $\Ttor{i}{M}{N}=0$ for all $i \le 0$, then the depth formula holds, i.e., in this setting, $\dpt{N} = \dpt{\tpp{M}{N}}$.

We record preliminary results in section 2 and give a proof of \thmref{gdepth1} in section 3; see Corollary \ref{gdepth1cor}. Section 3 is also devoted to other applications of our argument. For example we analyze the depth of $\G$-relative homology modules and obtain a necessary condition for the depth formula to hold; see Theorem \ref{thm4}. As another application, we obtain a class of rings over which absolute homology and $\G$-relative homology behave differently. This leads us to the content of the next result; see also \corref{211}.

\begin{prp} \label{prp:omoshiroi} Assume $R$ is a two-dimensional Gorenstein normal local domain. If $I$ is an ideal of $R$, then $\GTor{i}{I}{I}=0$ for all $i\ge 1$. In particular, if $R$ is not regular, then $\GTor{i}{\m}{\m}=0\neq \Tor{i}{\m}{\m}$ for all $i\ge 1$.
\end{prp}

\section{A $\G$-relative derived depth formula}

We start by recalling several definitions and terminology from \cite{MAsMBr69,LLAAMr02,AIc07}. Throughout we use homological notation for complexes of $R$-modules.

\begin{ipg} We say that a complex is \emph{acyclic} if it has zero
homology. A morphism of complexes that induces an isomorphism in homology is marked by the symbol `$\qis$'.

The soft truncation below at $n$ of a complex $T$, denoted by $\Tsb{n}{T}$, is the complex defined as
\begin{equation*}
  (\Tsb{n}{T})_i =
  \begin{cases}
    T_i &\text{for $i > n$}\,,\\
    \Ker{(T_n \to T_{n-1})} &\text{for $i = n$}\,,\\
    0 &\text{for $i < n$}\,.
  \end{cases}
\end{equation*}

The depth of a complex $T$ of $R$-modules is defined as:
\begin{equation*}
  \dpt{T} = \inf\setof{i\in\ZZ}{\H[-i]{\RHom{k}{T}} \ne 0}.
\end{equation*}
Here $k$ is the residue field of $R$, and the derived Hom complex,
$\RHom{k}{T}$, can be computed by using a
DG-injective resolution $T \qra I$ in the sense of \cite{LLAHBF91}.
We note that $\dpt{0}=\infty$.
\end{ipg}

\begin{bfhpg}[Tate homology] \label{TH}
  An acyclic complex $T$ of free $R$-modules is called \emph{totally
    acyclic} if the dual complex $\Hom{T}{R}$ is acyclic. An
  $R$-module $G$ is called \emph{totally reflexive} if there exists
  such a totally acyclic complex $T$ with $G\is \Coker{(T_1\to T_0)}$.

  The \emph{$\G$-dimension} of an $R$-module $M$, written $\Gdim{M}$,
  is the minimal length of a resolution of $M$ by totally reflexive
  modules (for the zero-module it is $-\infty$). If $\Gdim{M}<\infty$, then one has the \emph{Auslander--Bridger formula}; see \cite[4.13(b)]{MAsMBr69}.
  \begin{equation}\notag{}
    \Gdim{M}+\dpt{M} = \dptR
  \end{equation}

  A \emph{complete resolution} of an $R$-module $M$ is a diagram
  \begin{equation*}
    T \xra{\tau} P \qra M,
  \end{equation*}
  where $P \qra M$ is a projective resolution, $T$ is a totally
  acyclic complex of free $R$-modules, and $\tau_i$ is an isomorphism
  for $i \gg 0$. An $R$-module has finite $\G$-dimension if and only
  if it has a complete resolution; moreover, if a module has a
  complete resolution, then it has one with $\tau$ surjective; see
  \cite[3.1]{LLAAMr02}.

  Let $M$ be an $R$-module with a complete resolution $T \to P \to
  M$. For an $R$-module $N$, \emph{Tate homology} of $M$ and $N$ is
  defined as
  \begin{equation*}
    \Ttor{i}{M}{N} = \H[i]{\tp{T}{N}} \ \text{for } i\in\ZZ\:.
  \end{equation*}
\end{bfhpg}

\begin{bfhpg}[Relative homology] \label{RH}
A sequence $\eta$ of $R$-modules is called \emph{$\G$-proper} if the
  induced sequence $\Hom{G}{\eta}$ is exact for every totally
  reflexive $R$-module $G$.

  A \emph{$\G$-proper resolution} of an $R$-module $M$ is a resolution
  $L \qra M$ by totally reflexive $R$-modules such that the augmented
  resolution $L^+$ is a $\G$-proper sequence. Every module of finite
  $\G$-dimension has a $\G$-proper resolution $L\qra M$, even one with
  $L_i$ projective for all $i>0$ and $L_i=0$ for $i>\Gdim(M)$; see
  \cite[3.1]{LLAAMr02}.

  Let $M$ be an $R$-module with a $\G$-proper resolution $L \qra
  M$. Any two $\G$-proper resolutions of $M$ are homotopy equivalent; see
  \cite[4.3]{LLAAMr02}. For an $R$-module $N$, we set
  \begin{equation*}
    \GLtp{M}{N} \deq \tp{L}{N}.
  \end{equation*}
Note that $\GLtp{M}{N}$ is unique up to isomorphism in the
  derived category over $R$.  The
  \emph{$\G$-relative} homology of $M$ and $N$ is defined as
  \begin{equation*}
    \GTor{i}{M}{N} = \H[i]{\GLtp{M}{N}} \ \text{for } i\in\ZZ.
  \end{equation*}

\vspace{0.7cm}

The following facts will be used several times:
\begin{enumerate}[\rm(i)]
\item Assume $\Gdim(M)<\infty$. Then $\GTor{i}{M}{N}=0$ for all $i>\Gdim(M)$. Furthermore, if
$\GTor{i}{M}{N}=0$ for all $i\ge 1$, then $\GLtp{M}{N} \dqis \tp{M}{N}$.
\item If $\pd(M)<\infty$, then $\GTor{i}{M}{N} \is \Tor{i}{M}{N}$ for all $i\geq 0$.
\end{enumerate}
\end{bfhpg}

\begin{thm}
  \label{thm:gdepth} Let $M$ be an $R$-module of finite $\G$-dimension. For every
  $R$-module $N$ with $\Ttor{i}{M}{N}=0$ for all $i \le 0$, there is an equality
\begin{equation*}
\dpt{M}+\dpt{N}= \dptR+ \dpt{\GLtpp{M}{N}}.
  \end{equation*}
\end{thm}

\begin{prf*}
  We can assume that $M$ is non-zero, otherwise both sides of the
  equality are $\infty$. Choose a complete resolution $T \xra{\tau} P
  \qra M$ with $\tau$ surjective, and set $K=\Ker{\tau}$. The exact
  sequence $0 \to K \to T \to P \to 0$ is degree-wise split, and so
  the next sequence is exact as well,
  \begin{equation*}
    0\lra \tp{K}{N} \lra \tp{T}{N} \lra \tp{P}{N} \lra 0.
  \end{equation*}
  There is a degree-wise split exact
  sequence $0 \to \Shift[-1]{L} \to \Tsb{-1}{T} \to P \to 0$ such that
  $L \qra M$ is a $\G$-proper resolution; see \cite[3.8]{LLAAMr02}. Consider the  exact sequence
  \begin{equation*}
    0 \lra \tp{\Shift[-1]{L}}{N} \lra \tp{\Tsb{-1}{T}}{N} \lra
    \tp{P}{N} \lra 0.
  \end{equation*}
  The assumption that $\Ttor{i}{M}{N}$ vanishes for all $i \le 0$
  ensures that the embedding $\tp{\Tsb{-1}{T}}{N} \to \tp{T}{N}$ is a
  quasi-isomorphism. Consider the following commutative diagram with
  exact rows:
  \begin{equation*}
    \xymatrix{
      0 \ar[r] & \tp{\Shift[-1]{L}}{N} \ar[r]
      & \tp{\Tsb{-1}{T}}{N} \ar[r] \ar[d]^-\qis
      & \tp{P}{N} \ar@{=}[d] \ar[r] & 0\\
      0 \ar[r] &  \tp{K}{N} \ar[r] & \tp{T}{N} \ar[r] & \tp{P}{N} \ar[r] & 0
    }
  \end{equation*}
  It yields a quasi-isomorphism $\tp{\Shift[-1]{L}}{N} \to
  \tp{K}{N}$. Let $N \qra I$ be an injective resolution; the induced
  morphism $\tp{\Shift[]{K}}{N} \to \tp{\Shift[]{K}}{I}$ is a
  quasi-isomorphim, as $K$ is a complex of projective modules; see
  \cite[2.14]{CFH-06}. As $K$ is a bounded above,
  $\tp{\Shift[]{K}}{I}$ is a bounded above complex of injective
  modules. Hence the composite
  \begin{equation*}
    \tp{L}{N} \qra \tp{\Shift[]{K}}{N} \qra \tp{\Shift[]{K}}{I}
  \end{equation*}
  is a DG-injective resolution, so one has
  \begin{equation*}
    \dpt{\GLtpp{M}{N}} \deq
    \inf\setof{i\in\ZZ}{\H[-i]{\Hom{k}{\tp{\Shift{K}}{I}}}\ne 0}\:.
  \end{equation*}
  The first isomorphism below follows from \cite[1.1]{LWCDAJ15},
  and the second uses that the $R$-action on $\Hom{k}{I}$ factors
  through $k$,
  \begin{equation*}
    \Hom{k}{\tp{\Shift{K}}{I}} \dis
    \tp{\Hom{k}{I}}{\Shift{K}} \dis
    \tp[k]{\Hom{k}{I}}{\tpp{k}{\Shift{K}}}\:.
  \end{equation*}
  Now the K\"unneth formula and the definition of depth yield
  \begin{equation*}
    \dpt{\GLtpp{M}{N}}
    \deq \dpt{N}
    + \inf\setof{i}{\H[-i]{\tp{k}{\Shift{K}}}\ne 0}\:.
  \end{equation*}
  For $N=R$, the equality reads $\dpt{M} = \dptR +
  \inf\setof{i}{\H[-i]{\tp{k}{\Shift{K}}}\ne 0}$, and the desired
  equality follows.
\end{prf*}

\section{Applications of \thmref{gdepth}}
We give three main applications of \thmref{gdepth}. The first one is \thmref{gdepth1}, which we advertised in the introduction; see Corollary \ref{gdepth1cor}. Our second application is a necessary condition for the depth formula in terms of the vanishing of $\G$-relative homology; see Theorem \ref{thm4}. This application, in particular, gives a class of rings over which $\GTor{i}{I}{I}$ vanishes for all ideals $I$ and for all positive integers $i$; see Corollary \ref{cor:211}. As a final application of \thmref{gdepth}, in Corollary \ref{cor:gadepth2}, we prove a  $\G$-relative version of Jorgensen's dependency formula \cite[2.2]{DAJ99}.

\subsection*{A proof of \thmref{gdepth1}}

The next corollary is advertised as \thmref{gdepth1} in the introduction.

\begin{cor} \label{gdepth1cor} Let $M$ and $N$ be $R$-modules. Assume $M$ has  finite $\G$-dimension. Then the depth formula holds provided the following holds:
\begin{enumerate}[\rm(i)]
\item $\Ttor{i}{M}{N}=0$ for all $i \le 0$.
\item $\GTor{i}{M}{N}=0$ for all $i \ge 1$.
\end{enumerate}
\end{cor}

\begin{proof} Since $\GTor{i}{M}{N}=0$ for all $i \ge 1$, one has $\GLtp{M}{N} \dqis \tp{M}{N}$; see \ref{RH}(i) Hence it follows from \thmref{gdepth} that the depth formula holds.
\end{proof}

Recall that, if $M$ is totally reflexive, then $\GTor{i}{M}{N}=0$ for all $i \ge 1$. So one has the following special case of Corollary \ref{gdepth1cor}:

\begin{cor} \label{cor:general1} Let $M$ be a totally reflexive $R$-module. If $N$ is an $R$-module such that $\Ttor{i}{M}{N}=0$ for all $i\le 0$, then $\dpt{N} =\dpt{\tpp{M}{N}}$.
\end{cor}

We give an example to show that the assumption of vanishing of all negative Tate
homology in \corref{general1} cannot be removed in general.

\begin{exa} \label{exa:1} Let $R = \CC[\![x,y]\!]/(xy)$, $M = R/(x)$ and $N=R/(y)$. Then $M$
is totally reflexive and $\dpt{N}=1 \neq  0= \dpt{\tpp{M}{N}}$. Note one has $\Ttor{i}{M}{N} \cong k$ for all even and negative integers $i$.
\end{exa}

The next example shows the vanishing of negative and positive Tate homology are two distinct conditions in general. In particular it points out that there are settings for which \thmref{gdepth1} applies while the result of Christensen and Jorgensen, namely \thmref{recover}, does not.

\begin{exa} \label{exa:2} It follows from a result of Jorgensen and \c{S}ega \cite[4.1]{DAJLMS05} that there exists an artinian Gorenstein local ring $R$ and finitely generated $R$-modules $M$ and $N$ such that
  \begin{equation*}
    \Ttor{i}{M}{N}\,
    \begin{cases}
      = 0 & \text{ for $i > 0$ }\\
      \ne 0 & \text{ for $i < 0$}.
    \end{cases}
  \end{equation*}
  Let $T \to P \to M$ be a complete resolution of $M$. Applying $(-)^* =
  \Hom{-}{R}$, one obtains:
  \begin{equation}\tag{\ref{exa:2}.1}
    \tpp{T}{N}^* \is \Hom{T}{N^*} \is
    \Hom{\Hom{T^*}{R}}{N^*} \is \tp{T^*}{N^*}.
  \end{equation}
The first and the third isomorphisms of (\ref{exa:2}.1) follow from the adjointness and the Hom evaluation, respectively; see \cite[II.5.2 and VI.5.2]{CE}. On the other hand the second isomorphism in (\ref{exa:2}.1) holds since $T$ is a complex of finitely generated projective $R$-modules. Now set $Y=N^*$ and let $X$ be the
  Auslander transpose of $M$, i.e., $X =\Coker{(T_0^* \to
    T_1^*)}$. In view of (\ref{exa:2}.1), for all $i\in \ZZ$, one has the following isomorphisms:
\begin{equation}\tag{\ref{exa:2}.2}
 \H[i]{\tp{T^*}{N^*}} \cong  \H[i]{\tpp{T}{N}^*} \cong (\H[-i]{T\otimes_{R}N})^* \cong \Ttor{-i}{M}{N}^* .
  \end{equation}
Consequently one deduces from (\ref{exa:2}.2) that:
  \begin{equation*}
    \Ttor{i}{X}{Y} = \H[i-1]{\tp{T^*}{N^*}} \cong \Ttor{-i+1}{M}{N}^*= \,
    \begin{cases}
      \ne 0 & \text{ for $i > 1$ }\\
      = 0 & \text{ for $i < 1$}.
    \end{cases}
  \end{equation*}
\end{exa}

One can obtain a slightly modified version of Corollary \ref{gdepth1cor} in case $\Ttor{i}{M}{N}$ vanishes for all $i\le \Gdim(M)$. To establish this we need a result of Iacob \cite{AIc07}, who showed absolute, relative, and Tate homology fit together in an exact sequence.

\begin{ipg}(Iacob \thmcite[1]{AIc07}) \label{Iac}
Let $M$ be an $R$-module of finite $\G$-dimension. For every $R$-module $N$
there is an exact sequence:
\begin{equation*}
  \cdots \lra \GTor{2}{M}{N} \lra \Ttor{1}{M}{N} \lra
  \Tor{1}{M}{N} \lra \GTor{1}{M}{N} \lra 0.
\end{equation*}
\end{ipg}

\begin{cor} \label{cor:general} Let $M$ and $N$ be $R$-modules. Assume $M$ has finite $\G$-dimension $g$. Assume further the following conditions hold:
\begin{enumerate}[\rm(i)]
\item $\Ttor{i}{M}{N}=0$  for all  $i\le g$.
\item If $g\ge 1$, assume $\Tor{i}{M}{N}=0$  for all $i=1, \ldots, g$.
\end{enumerate}
Then $\dpt{M} + \dpt{N} \deq  \dptR+ \dpt{\tpp{M}{N}}$.
\end{cor}

\begin{prf*} The case where $g=0$ is \corref{general1}. Hence assume $g\ge 1$. Since $\Ttor{i}{M}{N}=0$  for all  $i\le g$, it follows from \ref{Iac} that $\GTor{i}{M}{N} \is \Tor{i}{M}{N}$ for all $i=1, \ldots, g$. Therefore, by the hypothesis, one has
$\GTor{i}{M}{N}=0$ for all  $i=1, \ldots, g$. This shows $\GTor{i}{M}{N}=0$ for all $i\ge 1$, and  result follows from Corollary \ref{gdepth1cor}.
\end{prf*}

In passing we record a different proof of Corollary \ref{cor:general} that does not make use of our  results:

\textit{An alternative proof of Corollary \ref{cor:general}}:
We argue by induction on $g$. Assume $g=0$. Then $\Gdim(\Tr M)=0$, where $\Tr M$ is the (Auslander's) transpose of $M$. The following isomorphisms hold for all $i\le 0$; see \cite[4.4.7]{AvBu}.
 \[\begin{array}{rl}\tag{\ref{cor:general}.1}
\Ttor{i}{M}{N}&\cong\TExt{-i-1}{M^*}{N}\\
&\cong\TExt{-i+1}{\Tr M}{N}\\
&\cong\Ext{-i+1}{\Tr M}{N}.
\end{array}\]
By \cite[2.6]{MAsMBr69}, there is an exact sequence of the form:
\begin{equation}\tag{\ref{cor:general}.2}
\Ext{1}{\Tr M}{N}\hookrightarrow M\otimes_RN\lra\Hom{M^*}{N}\twoheadrightarrow \Ext{2}{\Tr M}{N}
\end{equation}
In view of the hypothesis, it follows from (\ref{cor:general}.1) and (\ref{cor:general}.2) that there exists an isomorphism $M\otimes_RN\cong\Hom{M^*}{N}$ and $\Ext{i}{M^*}{N}=0$ for all $i\ge 1$. Hence the assertion follows from \cite[4.1]{AY} and the Auslander -- Bridger formula.

Next assume $g\ge 1$. There is an exact sequence of the form:
\begin{equation}\tag{\ref{cor:general}.3}
0\lra M\lra P\lra X\lra0
\end{equation}
where $\Gdim(X)=0$ and $\pd(P)=\Gdim(M)=g<\infty$; see \cite[3.3]{LWCI}. Therefore $\dpt{P}=\dpt{M}$. Moreover, since $\Ttor{i}{P}{N}=0$ for all $i\in \ZZ$, it follows from (\ref{cor:general}.3) that
$\Ttor{i}{M}{N}\cong\Ttor{i+1}{X}{N}$ for all $i\in\ZZ$. This implies $\Ttor{i}{X}{N}=0$ for all $i\le g$. So, by the case where $g=0$, we conclude $\dpt{N} \deq  \dpt{\tpp{X}{N}}$.
Furthermore, as $\Ttor{i}{X}{N}\cong \Tor{i}{X}{N}$ for all $i\ge 1$, we have $\Tor{i}{X}{N}=0$ for all $i=1, \ldots, g$. Thus it follows from (\ref{cor:general}.3) that $\Tor{i}{P}{N}=0$ for all $i\ge 1$, and the next equality holds; see  \cite[1.2]{Au} (or \ref{ADF}).
\begin{equation}\tag{\ref{cor:general}.4}
\dpt{P} + \dpt{N} \deq  \dptR+ \dpt{\tpp{P}{N}}.
\end{equation}
It now suffices to see $\dpt(P\otimes_RN) = \dpt(M\otimes_RN)$. The sequence (\ref{cor:general}.3) induces the exact sequence:
\begin{equation}\tag{\ref{cor:general}.5}
0\lra M\otimes_RN\lra P\otimes_RN\lra X\otimes_RN\lra0.
\end{equation}
We have $\dpt(P\otimes_RN)=\dpt{N}-\pd(P)$ by (\ref{cor:general}.4). This shows that $\dpt(P\otimes_RN)<\dpt(N)=\dpt(X\otimes_RN)$. Using the depth lemma with (\ref{cor:general}.5), we see $\dpt(P\otimes_RN) = \dpt(M\otimes_RN)$.

\subsection*{A necessary condition for the depth formula}

Necessary conditions for the depth formula via the vanishing of absolute homology have been studied previously; see \cite[1.2]{CeD}. In Theorem \ref{thm4} we obtain such a necessary condition  that makes use of $\G$-relative homology. As a consequence, we prove in Corollary \ref{cor:211} that $\G$-relative homology of proper ideals vanishes diametrically opposed to absolute homology over two dimensional Gorenstein normal domains. First we record some preliminaries.

\begin{ipg}\label{lem:GTorseq} (Avramov and Martsinkovsky \cite{LLAAMr02})
 \label{lem:seq}
  Let $0 \to M' \to M \to M'' \to 0$ be a $\G$-proper sequence of
  $R$-modules of finite $\G$-dimension.  For every $R$-module $N$
  there is an exact sequence:
  \begin{align*}
    \label{eq:seq}
    \cdots \to \GTor{i+1}{M''}{N}
    \to \GTor{i}{M'}{N} \to \GTor{i}{M}{N} \to \GTor{i}{M''}{N} &\to \\
    \cdots \to \GTor{1}{M''}{N}\to \tp{M'}{N} \to \tp{M}{N} \to
    \tp{M''}{N} &\to 0.
  \end{align*}
\end{ipg}

\begin{rmk} \label{good} Assume $M$ is an $R$-module of finite $\G$-dimension. It follows from \cite[3.1]{LLAAMr02} and \cite[4.1]{LLAAMr02} that there is a $\G$-proper exact sequence
\begin{equation}\tag{\ref{good}.1}
0\rightarrow L\rightarrow X\rightarrow M\rightarrow0,
\end{equation}
where $\pd(L)<\infty$ and $\Gdim(X)=0$, i.e., $X$ is totally reflexive. Hence  \ref{lem:GTorseq} and (\ref{good}.1) yield the long exact sequence:
\begin{equation}  \tag{\ref{good}.2}
\cdots\rightarrow\GTor{i}{L}{N}\rightarrow\GTor{i}{X}{N}\rightarrow\GTor{i}{M}{N}\rightarrow\cdots.
\end{equation}
Note that $\GTor{i}{X}{N}=0$ and $\Tor{i}{L}{N}\cong\GTor{i}{L}{N}$ for all $i\ge 1$; see \ref{RH}(i,ii). Therefore, by (\ref{good}.2), the following isomorphims hold:
\begin{equation}\tag{\ref{good}.3}
\Tor{i}{L}{N}\cong\GTor{i}{L}{N}\cong\GTor{i+1}{M}{N}
\text{ for all } i\ge 1,
\end{equation}
and we have the following exact sequence:
\begin{equation}\tag{\ref{good}.4}
0\rightarrow\GTor{1}{M}{N}\rightarrow L\otimes_RN\rightarrow X\otimes_RN\rightarrow M\otimes_RN\rightarrow0.\qed
\end{equation}
\end{rmk}

\begin{ipg} (Auslander \cite[1.2]{Au}) \label{ADF}
Let $M$ and $N$ be $R$-modules such that $M$ has finite projective dimension. Set $q=\sup{\setof{i}{\Tor{i}{M}{N}\ne 0}}$. If $\dpt{\Tor{q}{M}{N}}\leq 1$ or $q=0$, then the following equality holds:
\begin{equation*}
    \dpt{M}+ \dpt{N} = \dptR + \dpt{\Tor{q}{M}{N}} -q.
  \end{equation*}
\end{ipg}

\vspace{0.1cm}

\begin{thm} \label{thm4}
 Let $M$ and $N$ be $R$-modules such that $M$ has finite $\G$-dimension. Assume that the following conditions hold:
\begin{enumerate}[\rm(i)]
\item $\dpt \GTor{i}{M}{N}\in \{0, \infty\}$ for all $i\ge 1$.
\item $\dpt{M} + \dpt{N} \deq    \dptR+ \dpt{\tpp{M}{N}}$.
\end{enumerate}
Then $\GTor{i}{M}{N}=0$ for all $i\ge 1$.
\end{thm}

\begin{prf*} Set $g=\Gdim(M)$. There is nothing to prove if $g=0$ so we may assume $g\ge 1$; see \ref{RH}(i). It follows from the depth lemma and (\ref{good}.1) that $g=\pd L+1$. Also note, by (ii), one has:
\begin{equation}\tag{\ref{thm4}.1}
g-\dpt N=-\dpt(M\otimes_RN).
\end{equation}

We will first prove $\GTor{i}{M}{N}=0$ for all $i\ge 2$. Consequently, if $g=1$,  there is nothing to prove. Hence assume $g\ge 2$.

Set $w=\sup\{i\in \ZZ \mid\Tor{i}{L}{N}\neq0\}$ and suppose $w \neq 0$.
As $w\le \pd L=g-1$ and $\Tor{w}{L}{N} \is \GTor{w+1}{M}{N}$, one has from (i) that $\dpt \Tor{w}{L}{N}=0$. Now \ref{ADF} applied to the pair $(L,N)$ yields:
\begin{equation}\tag{\ref{thm4}.2}
w=\pd L-\dpt N=g-\dpt N-1.
\end{equation}
Now (\ref{thm4}.1) and (\ref{thm4}.2) give $w=-\dpt(M\otimes_RN)-1<0$, i.e., a contradiction. So $w=0$ and hence $\GTor{i}{M}{N}=0$
for all $i\ge 2$ by (\ref{good}.3).

Next we will prove $\GTor{1}{M}{N}=0$. Suppose $\GTor{1}{M}{N}\neq0$.
Note, by (i), one has $\dpt \GTor{1}{M}{N}=0$. Thus one concludes from (\ref{good}.4) that $\dpt (L\otimes_RN)=0$. As $w=0$, the pair $(L,N)$ satisfies the depth formula of \ref{ADF}, and this yields $\dpt N=\pd L=g-1$, i.e., $\dpt N-g=-1$. Therefore, by (\ref{thm4}.1), one has $\dpt(M\otimes_RN)=-1$, i.e., a contradiction. Consequently $\GTor{1}{M}{N}$ must vanish and this finishes the proof.
\end{prf*}

The following is an immediate consequence of Theorem \ref{thm4}.

\begin{cor} \label{GTorref}
Let $M$ and $N$ be nonzero $R$-modules such that the depth formula holds, i.e., $\dpt{M} + \dpt{N} \deq  \dptR+ \dpt{\tpp{M}{N}}$. If $M$ has finite $\G$-dimension $g$ and $\GTor{i}{M}{N}$ has finite length for all $i=1, \ldots, g$, then $\GTor{i}{M}{N}=0$ for all $i\ge 1$.
\end{cor}

We note next that $\G$-relative homology localizes. This allows us to give another necessary condition for the depth formula; see Corollary \ref{cor:2100}.

\begin{rmk} \label{rmkloc} Let $M$ be an $R$-module of finite $\G$-dimension and let $N$ be an $R$-module. Then $\GTor{i}{M}{N}_\p \dis \GTor[R_\p]{i}{M_\p}{N_\p}$ for all primes $\p$ and  all $i\in\ZZ$. To see this, let $\p$ be a prime ideal that is in the support of $M$ and $N$. Let $L\qra M$ be a $\G$-proper resolution of $M$ such that $L_i$ is projective for
  all $i\ge 1$ and $L_i=0$ for all $i\gg 0$. For all $i\ge 1$,  the syzygy
  $\Coker{((L_{i+1})_\p \to (L_{i})_\p)}$ is a module of
  finite projective dimension, so that $L_\p \qra M_\p$ is a $\G$-proper
  resolution of $M_\p$ by \cite[2.1]{CFH-06}. The localization isomorphism
  for $\G$-relative homology now follows from the standard isomorphism
  $\tpp{L}{N}_\p \is \tp[R_\p]{L_\p}{N_\p}$.
\end{rmk}

\begin{cor} \label{cor:2100} Let $M$ and $N$ be $R$-modules such that $M$ has finite $\G$-dimension. Assume $M_{\p}$ is a totally reflexive $R_{\p}$-module for all prime ideals $\p\ne \m$. If the depth formula holds, i.e., $\dpt{M} + \dpt{N} \deq  \dptR+ \dpt{\tpp{M}{N}}$, then $\GTor{i}{M}{N}=0$ for all $i\ge 1$.
\end{cor}

\begin{prf*} For each prime ideal $\p\ne \m$, we have $\GTor{i}{M}{N}_\p \dis \GTor[R_\p]{i}{M_\p}{N_\p}=0$ for all $i\ge 1$; see \ref{RH}. Therefore $\dpt \GTor{i}{M}{N}\in \{0, \infty\}$ for all $i\ge 1$. Hence the result follows from Theorem \ref{thm4}.
\end{prf*}

The \emph{torsion submodule} $\ttpp R M$ of an $R$-module $M$ is the kernel of the natural homomorphism $M\to \text{Q}(R)\otimes_RM$, where $\text{Q}(R)$ is the total quotient ring of $R$. $M$ is said to have \emph{torsion} if $\ttpp RM\ne 0$, and said to be \emph{torsion-free} if $\ttpp RM
= 0$.

\begin{thm} \label{prp:217} Let $R$ be a two-dimensional Gorenstein normal local domain. If $I$ and $J$ are ideals of $R$, and $\tp{I}{J}$ has torsion, then $\GTor{i}{I}{J}=0$ for all $i\ge 1$.
\end{thm}

\begin{prf*} One may assume $I$ is nonzero. Notice that $R$ has positive depth, so one has $\dpt I \ge 1$. If $\dpt I=2$, then $I$ is totally reflexive, and $\GTor{i}{I}{J}=0$ for all $i\ge 1$; see \ref{RH}. Thus one may assume $\dpt I=1$. Similarly one may also assume $J$ is a nonzero ideal of depth one.

Notice, by Remark \ref{rmkloc}, $\GTor{i}{I}{J}$ has finite length for all $i\ge 1$. Hence, to establish the required conclusion, it suffices to prove, by Theorem \ref{thm4}, that the depth formula for $(I,J)$ holds, i.e., it suffices to prove $\dpt(I\otimes_{R}J)=0$.

Set $T=I\otimes_{R}J$ and suppose $\dpt T \ne 0$. For a prime ideal $\p$ of $R$  that has height one, $I_{\p}$ and $J_{\p}$ are maximal Cohen-Macaulay modules over the regular ring $R_{\p}$, and hence both $I_{\p}$ and $J_{\p}$ are free. Thus one has $\dpt[R_\p]T_{\p}\geq \min\{1, \dim R_{\p}\}$ for all prime ideals $\p$ of $R$. Now let $x$ be a non-zero divisor on $R$. Suppose $x$ is contained in the set of zero-divisors of $T$. Then $x\in \mathfrak{q}$ for some associated prime $\mathfrak{q}$ of $T$, and hence $0=\dpt[R_\mathfrak{q}]T_{\mathfrak{q}}\geq \min\{1, \dim R_{\mathfrak{q}}\}$. This shows $\mathfrak{q}$ is a minimal prime ideal of $R$ and yields a contradiction since $\mathfrak{q}$ contains a non-zero divisor. So $x$ is a non-zero divisor on $T$. In other words $T$ is torsion-free. This gives a contradiction to our assumption that $T$ has torsion. Therefore $\dpt T =0$.
\end{prf*}

Tensor products of finitely generated nonzero modules tend to have torsion in general. Therefore -- thanks to \prpref{217} -- examples of ideals $I$ and $J$ for which $\GTor{\ast}{I}{J}$ vanish are abundant. In particular one has:

\begin{cor} \label{cor:211} Let $R$ be a two-dimensional Gorenstein normal local domain. If $I$ is an ideal of $R$, then $\GTor{i}{I}{I}=0$ for all $i\ge 1$. In particular, if $R$ is not regular, one has $\GTor{i}{\m}{\m}=0\ne \Tor{i}{\m}{\m}$ for all $i\ge 1$.
\end{cor}

\begin{prf*} One may assume $I\ne 0$. If $\tp{I}{I}$ is torsion-free, then $\tp{I}{I} \is I^2$ so, by counting the minimal number of generators of $I$, one can see that $I$ must be principal, i.e., $I \is R$; see \cite[page 467]{HW}. Thus one may assume $\tp{I}{I}$ has torsion. In that case the result follows from Theorem \ref{prp:217}.
\end{prf*}

\corref{211} can be used to show $\GTor{\ast}{I}{I}=0\ne \Tor{\ast}{I}{I}$ for some nonmaximal ideals $I$. We give such an example next.

\begin{exa} \label{exa:3} Let $R=\CC[\![x,y,z]\!]/(xy-z^2)$ and let $I=(x,z)$ be the ideal of $R$ generated by $x$ and $z$. Note $R$ is a two-dimensional Gorenstein normal local domain that is not regular, so one has $\GTor{i}{I}{I}=0$ for all $i\ge 1$; see \corref{211}. If $\Tor{n}{I}{I}=0$ for some nonnegative integer $n$, then $\Tor{i}{I}{I}=0$ for all $i \ge n$, and this forces $I$ to have finite projective dimension; see \cite[5.1]{HW} and \cite[1.9]{HW2}.
However $R/I \cong \CC[\![y]\!]$, and hence $\dpt R/I=1$ and $\dpt I=2$, i.e., $I$ is maximal Cohen-Macaulay. Therefore $I$ does not have finite projective dimension, and so $\Tor{i}{I}{I}\ne 0$ for all $i\ge 0$.
\end{exa}

\subsection*{On Jorgensen's dependency formula}

Let $M$ be an $R$-module that has finite complete intersection dimension \cite{AGP} (e.g., $R$ is a complete intersection ring). If $N$ is an $R$-module with $\Tor{i}{M}{N}=0$ for all $i\gg 0$ and $q=\sup{\setof{i}{\Tor{i}{M}{N}\ne 0}}$, Jorgensen \cite[2.2]{DAJ99} proved an equality, which he referred to as the \emph{dependency formula}:
  \begin{equation*} q=
    \sup_{}\{\dptR_\p - \dpt[R_\p]{M_\p} - \dpt[R_\p]{N_\p} \; \arrowvert\; \p\in \Supp(M)\cap \Supp(N)\}.
\end{equation*}
Our final application of \thmref{gdepth} is a $\G$-relative version of Jorgensen's dependency formula, which we reach in  \corref{gadepth2}.

\begin{cor}
  \label{cor:gadepth}
  Let $M$ be an $R$-module of finite $\G$-dimension and let $N$ be an
  $R$-module with $\Ttor{i}{M}{N}=0$ for all $i \le 0$. Set
  $s=\sup{\setof{i}{\GTor{i}{M}{N}\ne 0}}$. Then the following conditions hold:
\begin{enumerate}[\rm(i)]
\item $\dpt{M} + \dpt{N}  \ge  \dptR - s$.
\item $\dpt{M} + \dpt{N} = \dptR-s$ if and only if $\dpt{\GTor{s}{M}{N}}=0$.
\item If $s=0$ or $\dpt{\GTor{s}{M}{N}}\le 1$, then one
  has
  \begin{equation*}
  \dpt{M}+ \dpt{N} = \dptR + \dpt{\GTor{s}{M}{N}} -s.
  \end{equation*}
\end{enumerate}
\end{cor}

\begin{prf*}
  By the definition of depth and \rmkcite[1.7.(2)]{LLAHBF91} one has
  \begin{equation*}
    \dpt{\GLtpp{M}{N}} \dge \inf\setof{i}{\H[-i]{\GLtp{M}{N}}\ne 0} = -s\:,
  \end{equation*}
  so (i) follows from \thmref{gdepth}. By noting $\GTor{s}{M}{N} = \H[s]{\GLtp{M}{N}}$, one can see that (ii) follows from \cite[1.5.(3)]{HBFSIn03}.

Finally if $s=0$ or $\dpt{\GTor{s}{M}{N}}\le 1$, then one has
\begin{equation*}
\dpt{\GTor{s}{M}{N}} - s \le \dpt{\GTor{i}{M}{N}}- i, \text{ for all } i\le s,
  \end{equation*}
and the formula in (iii) follows from \thmref{gdepth} and \thmcite[2.3]{SIn99}.
\end{prf*}

\begin{cor}
  \label{cor:gadepth2} Let $M$ and $N$ be $R$-modules such that $M$ has finite $\G$-dimension. Set
$m=\sup_{}\{\dptR_\p - \dpt[R_\p]{M_\p} - \dpt[R_\p]{N_\p} \; \arrowvert\; \p\in \Supp(M)\cap \Supp(N)\}$ and $s=\sup{\setof{i}{\GTor{i}{M}{N}\ne 0}}$.
Then $s \leq m$. If furthermore, $\Ttor{i}{M}{N}=0$ for all $i \le 0$, then  $s=m$.
\end{cor}

\begin{prf*} There is nothing to prove if $\Gdim(M)=0$. So assume $s\ge 1$ and $\Gdim(M)\ge 1$. Let $\p\in\Ass_R(\GTor{s}{M}{N})$. Then it follows from (\ref{good}.3) and (\ref{good}.4) that $\p\in\Ass_R(\Tor{s-1}{L}{N})$, $\sup\{i\mid\Tor{i}{L}{N}\neq0\}=s-1$ and $\Gdim(M_\p)\ge 1$. Using Theorem \ref{ADF} for the pair $(L_\p, N_\p)$, one has
\begin{equation*}\tag{\ref{cor:gadepth2}.1}
\dpt[R_\p]{L_\p}+\dpt[R_\p]{N_\p}=\dptR_\p-(s-1)=\dptR_\p-s+1.
\end{equation*}
Localizing (\ref{good}.1) at $\p$, and noting $\dpt[R_\p]{X_\p}=\dptR_\p>\dpt[R_\p]{M_\p}$, one concludes by the depth lemma that $\dpt[R_\p]{L_\p}=\dpt[R_\p]{M_\p}+1$. Using this equality in  (\ref{cor:gadepth2}.1), one obtains
$s=\dptR_\p-\dpt[R_\p]{M_\p}-\dpt[R_\p]{N_\p}$. This shows $s\le m$.
Now assume $\Ttor{i}{M}{N}=0$ for all $i \le 0$. Then it follows $\Ttor[R_\p]{i}{M_\p}{N_\p}=0$ for all prime ideals $\p$ of $R$, and for all $i \le 0$. This gives, by Remark \ref{rmkloc} and \corref{gadepth}, that
  \begin{equation*}
     s \ge \sup{\setof{i}{\GTor[R_\p]{i}{M_\p}{N_\p}\ne 0}}
    \ge \dptR_\p - \dpt[R_\p]{M_\p} - \dpt[R_\p]{N_\p}.
  \end{equation*}
Therefore $s\ge m$, and hence $s=m$ holds.
\end{prf*}

We do not know whether or not the vanishing of negative Tate homology in \corref{gadepth2} is necessary. Hence it seeems worth posing the next question.

\begin{que}  Let $M$ and $N$ be $R$-modules such that $M$ has finite $\G$-dimension. Set
$m=\sup_{}\{\dptR_\p - \dpt[R_\p]{M_\p} - \dpt[R_\p]{N_\p} \; \arrowvert\; \p\in \Supp(M)\cap \Supp(N)\}$ and $s=\sup{\setof{i}{\GTor{i}{M}{N}\ne 0}}$.
 Then must one have $s = m$?
\end{que}

\section*{Acknowledgments}
The research of Liang was partly supported by NSF of China with grant No. 11301240, NSF of Gansu Province with grant No.1506RJZA075 and SRF for ROCS, SEM. The research of Sadeghi was supported by a grant from IPM and a grant from Iran National Science Foundation (INSF) with grant No.94027467.

Part of this work was completed during a visit by Liang to the Texas Tech University; he is grateful for the kind hospitality of the Texas Tech Department of Mathematics and Statistics.

Celikbas is grateful to Jerzy Weyman for his support during the preparation of this article. The authors thank Lars Winther Christensen for his help, comments and suggestions and Greg Piepmeyer for his feedback on the manuscript. The authors also thank the referee for his/her careful reading and many useful comments.

\def\soft#1{\leavevmode\setbox0=\hbox{h}\dimen7=\ht0\advance \dimen7
  by-1ex\relax\if t#1\relax\rlap{\raise.6\dimen7
  \hbox{\kern.3ex\char'47}}#1\relax\else\if T#1\relax
  \rlap{\raise.5\dimen7\hbox{\kern1.3ex\char'47}}#1\relax \else\if
  d#1\relax\rlap{\raise.5\dimen7\hbox{\kern.9ex \char'47}}#1\relax\else\if
  D#1\relax\rlap{\raise.5\dimen7 \hbox{\kern1.4ex\char'47}}#1\relax\else\if
  l#1\relax \rlap{\raise.5\dimen7\hbox{\kern.4ex\char'47}}#1\relax \else\if
  L#1\relax\rlap{\raise.5\dimen7\hbox{\kern.7ex
  \char'47}}#1\relax\else\message{accent \string\soft \space #1 not
  defined!}#1\relax\fi\fi\fi\fi\fi\fi} \def\cprime{$'$}
  \providecommand{\arxiv}[2][AC]{\mbox{\href{http://arxiv.org/abs/#2}{\sf
  arXiv:#2 [math.#1]}}}
  \providecommand{\oldarxiv}[2][AC]{\mbox{\href{http://arxiv.org/abs/math/#2}{\sf
  arXiv:math/#2
  [math.#1]}}}\providecommand{\MR}[1]{\mbox{\href{http://www.ams.org/mathscinet-getitem?mr=#1}{#1}}}
  \renewcommand{\MR}[1]{\mbox{\href{http://www.ams.org/mathscinet-getitem?mr=#1}{#1}}}
\providecommand{\bysame}{\leavevmode\hbox to3em{\hrulefill}\thinspace}
\providecommand{\MR}{\relax\ifhmode\unskip\space\fi MR }
\providecommand{\MRhref}[2]{%
  \href{http://www.ams.org/mathscinet-getitem?mr=#1}{#2}
}
\providecommand{\href}[2]{#2}


\begin{thebibliography}{10}

\bibitem{AY}
Tokuji Araya and Yuji Yoshino, \emph{Remarks on a depth formula, a grade
inequality and a conjecture of Auslander}, Comm. Algebra \textbf{26} (1998),
3793-3806.

\bibitem{Au}
Maurice Auslander, \emph{Modules over unramified regular local rings},
Illinois J. Math. \textbf{5} (1961), 631--647.

\bibitem{MAsMBr69}
Maurice Auslander and Mark Bridger, \emph{Stable module theory}, Memoirs of the
  American Mathematical Society, No. 94, American Mathematical Society,
  Providence, R.I., 1969.

\bibitem{AvBu}
Luchezar L. Avramov, Ragnar O. Buchweitz, \emph{Support varieties and
cohomology over complete intersections}. Invent. Math. \textbf{142} (2000),
285-318.

\bibitem{LLAHBF91}
Luchezar L. Avramov and Hans-Bj{\o}rn Foxby, \emph{Homological dimensions of
  unbounded complexes}, J. Pure Appl. Algebra \textbf{71} (1991), 
  129--155.

\bibitem{AGP}
Luchezar L. Avramov, Vesselin N. Gasharov and Irena V. Peeva, \emph{Complete intersection dimension}, Inst. Hautes \'Etudes Sci. Publ. Math. \textbf{86} (1997), 67--114.

\bibitem{LLAAMr02}
Luchezar L. Avramov and Alex Martsinkovsky, \emph{Absolute, relative, and
  {T}ate cohomology of modules of finite {G}orenstein dimension}, Proc. London
  Math. Soc. \textbf{85} (2002), 393--440.

\bibitem{CE}
Henri Cartan and Samuel Eilenberg, \emph{Homological algebra}, Princeton
  Landmarks in Mathematics, Princeton University Press, Princeton, NJ, 1999,
  With an appendix by David A. Buchsbaum, Reprint of the 1956 original.

\bibitem{CeD}
Olgur Celikbas and Hailong Dao, \emph{Necessary conditions for the depth formula over Cohen-Macaulay local rings}, J. Pure Appl. Algebra \textbf{218} (2014), 522--530.

\bibitem{CFH-06}
Lars~Winther Christensen, Anders Frankild, and Henrik Holm, \emph{On
  {G}orenstein projective, injective and flat dimensions---{A} functorial
  description with applications}, J.~Algebra \textbf{302} (2006), 
  231--279.

\bibitem{LWCI} Lars~Winther Christensen and Srikanth Iyengar, Gorenstein dimension of modules over homomorphisms, J. Pure. appl. Algebra \textbf{208} (2007), 177 -- 188.

\bibitem{LWCDAJ15}
Lars~Winther Christensen and David~A. Jorgensen, \emph{Vanishing of {T}ate
  homology and depth formulas over local rings}, J. Pure Appl. Algebra
  \textbf{219} (2015), 464--481.

\bibitem{Dao}
Hailong Dao, \emph{Some homological properties of modules over a complete
              intersection, with applications}, Commutative Algebra, 335 -- 371, Springer, New York, 2013.

\bibitem{HBFSIn03}
Hans-Bj{\o}rn Foxby and Srikanth Iyengar, \emph{Depth and amplitude for
  unbounded complexes}, Commutative algebra (Grenoble/Lyon, 2001) (Providence,
  RI), Contemp. Math., vol. 331, Amer. Math. Soc., 2003, pp.~119--137.

\bibitem{HW2}
Craig~Huneke and Roger~Wiegand, \emph{Tensor products of modules, rigidity and local cohomology}, Math. Scand. \textbf{81} (1997), 161--183.

\bibitem{HW}
Craig~Huneke and Roger~Wiegand, \emph{Tensor products of modules and the rigidity of Tor}, Math. Ann. \textbf{299} (1994), 449--476. \texttt{Correction:} Math. Ann. \textbf{338} (2007), 291--293.


\bibitem{AIc07}
Alina Iacob, \emph{Absolute, Gorenstein, and Tate torsion modules}, Comm.
  Algebra \textbf{35} (2007), 1589--1606.

\bibitem{SIn99}
Srikanth Iyengar, \emph{Depth for complexes, and intersection theorems}, Math.
  Z. \textbf{230} (1999), 545--567.

\bibitem{DAJ99}
David~A. Jorgensen, \emph{A generalization of the Auslander-Buchsbaum
  formula}, J. Pure Appl. Algebra \textbf{144} (1999), 145--155.

\bibitem{DAJLMS05}
David~A. Jorgensen and Liana~M. {\c{S}}ega, \emph{Asymmetric complete
  resolutions and vanishing of Ext over Gorenstein rings}, Int. Math. Res.
  Not. \textbf{56} (2005), 3459--3477.


\end{thebibliography}
\end{document}